\documentclass[reqno]{amsart}
\usepackage{amsfonts,amsmath,amssymb}
\usepackage{color}
\usepackage{textcomp}
\usepackage{adjustbox}
\usepackage[T1]{fontenc}
\usepackage[utf8]{inputenc}
\usepackage{lmodern}
\usepackage[backref=page, colorlinks=true, citecolor=cyan, linkcolor=blue]{hyperref}

\numberwithin{equation}{section}

\newtheorem{thm}{Theorem}[section]
\newtheorem{defi}{Definition}[section]
\newtheorem{prop}{Proposition}[section]
\newtheorem{lem}{Lemma}[section]
\newtheorem{cor}{Corollary}[section]
\newtheorem{rem}{Remark}[section]

\newtheorem{ex}{Example}[section]

\newcommand{\R}{\mathbb{R}} 
\newcommand{\N}{\mathbb{N}}

\newcommand{\si}{s^{(0)}}

\makeatletter
\@namedef{subjclassname@2020}{%
  \textup{2020} Mathematics Subject Classification}
\makeatother

\title[Persistence criteria  for a chemostat model]{Persistence criteria  for a chemostat with variable nutrient input and variable washout with delayed response in growth}

\thanks{e-mail: mrodriguezcartabia@dm.uba.ar}

\begin{document} 

\maketitle

\centerline{\scshape Mauro Rodriguez Cartabia}
\medskip
{\footnotesize



\begin{abstract}
 We study a general single-species chemostat model  with non-autonomous input and washout, and a discrete-time delay between consumption of the nutrient and growth. The goal of this article is to provide sufficient and necessary conditions for persistence. The persistence criteria obtained here extend previous works that have studied  the autonomous model with delay or the non-autonomous model without delay. Furthermore, in the particular case of periodic nutrient input and the same periodic washout with delayed response in growth, the persistence criteria are average criteria. We remark in no  case we need to impose any restrictions on the size of the delay. 
\end{abstract}


\bigskip
 
%
 
\section{Introduction}

The chemostat is a manufactured device in  which  microorganisms are cultivated in a liquid by continuous  addition of fresh medium. The input contains in excess all the nutrients necessary for the growth of the organisms except for one that we denote  by \emph{nutrient}. Meanwhile,  the medium is continuously pumped out to keep the volume constant. The washout contains nutrient,  organisms,  and, in some cases, products produced by these  organisms.  We refer to \cite{smith1995theory} for a complete description.

To study the chemostat with  variable nutrient input and variable washout with  delayed response in growth we follow the model presented in \cite{amster2020existence} that combines non-autonomous (periodic) inputs and a delay  between  consumption and growth. In this article we firstly work with general inputs, without imposing any periodicity. Therefore, given a constant $\tau\geq 0$, consider the system  
\begin{equation}\label{sistemaprincipal}
  \left\{\begin{aligned}
  s'(t)&=\left( \si(t)-s(t)\right)D(t)-x(t)p(s(t)),& t\geq 0,\\
  x'(t)&=-D(t)x(t)+x(t-\tau)p(s(t-\tau))e^{-\int_{t-\tau}^tD(r)\, dr},& t\geq 0,
  \end{aligned}
  \right.
\end{equation}
with initial conditions
$$\left(s,x \right)\Big|_{[-\tau,0]}=\left(s^{in},x^{in} \right).$$
These initial conditions must be non-negative time functions defined over the interval $[-\tau,0]$.  Here $s$ represents the concentration of the nutrient and $x$ denotes the concentration of the organism. Furthermore, $\si$ and $D$ are non-negative time continuous functions representing the concentration of the input nutrient and the washout rate, respectively. As usual, we assume the function $p:[0,\infty)\to[0,\infty)$ is continuously differentiable with $p(0)=0$ and $p'(s)>0$ for all $s\geq 0$. This function can be seen as a generalization of the Monod model given by
$$p(s)=\frac{ms}{a+s}$$
where $m>0$ is the maximal growth rate and $a>0$ is the half-saturation constant. Here $\tau$ represents a constant delayed time between consumption of the nutrient and growth. Furthermore, $\exp(-\int_{t-\tau}^tD)$ is the biomass factor consumed at time $t-\tau$ that remains at time $t$. For an explanation of this model see \cite{ellermeyer2003theoretical, smith2011introduction}.

In this work we study the persistence of solutions  in the chemostat. This implies finding conditions, necessary and sufficient, to ensure that the concentration of the organism \emph{remains away from zero} for all  positive  times. To deepen in this concept see \cite{freedman1990persistence}. There are several studies of persistence in the chemostat, we focus on two. 
A  crucial work by Ellermayer was done in \cite{ellermeyer1994competition} where he proved persistence criteria for System \ref{sistemaprincipal}  in the case of $\si$ and $D$ constant positive functions. Furthermore, in \cite{ellermeyer2001persistence}, Ellermeyer \emph{et al.} provided persistence criteria for System \ref{sistemaprincipal} in the case of variable input and constant washout with instantaneous response in growth ($\tau=0$). 

Despite the relevance of this matter, there is a lack of analysis for persistence criteria in a more general case, such as the one proposed in System \ref{sistemaprincipal}. Furthermore, this lack includes the  particular case with $\si$ and $D$ periodic functions. Therefore the main result of this work is  Theorem \ref{persistencia} and, to the best of our knowledge, is the first one that provides persistence criteria for System \ref{sistemaprincipal}. To sketch it consider $z$ a solution of System \ref{sistemaprincipal} in the absence of organisms (called \emph{washout solution}).   The  theorem states that persistence occurs if and only if over large intervals the integral of $p(z)  $ weighted by a function $\varphi$ is larger than the integral of $D$. Here $\varphi$ depends on the system and is independent of any particular solution (see \eqref{funcionphi} for a precise definition). Analogously, in the case of periodic input and washout, $\varphi$ is periodic and Theorem \ref{persistenciaperiodica} states that the persistence occurs if and only if the average of $p(z)\varphi $ is larger than the average of $D$.

The rest of this paper is organized as follows. In Section \ref{seccionprincipal} we 
define the concepts mentioned above, and state Theorem \ref{persistencia}. In Section \ref{seccioncasos}, we relate  this theorem to certain particular cases, including  Theorem \ref{persistenciaperiodica}. We also emphasize the connection with previous works mentioned. Finally, Section \ref{pruebas} is devoted to proving Theorem \ref{persistencia} and Section \ref{otraspruebas} is devoted to proving the results of Section \ref{seccioncasos}.  

\section{Main result}\label{seccionprincipal}

Since $s$ and $x$ are concentrations, they must always be non-negative functions (we show this in Section \ref{pruebas}). Furthermore, we are interested in  organism-positive solutions. This motivates the following definition. 

\begin{defi}[Non-null initial conditions]
We say  $(s^{in},x^{in})$ is a \emph{non-null initial condition}  if its  time functions are non-negative, and  either $x^{in}(0)>0$ or there exists $t_*\in[-\tau,0]$ such that $s^{in}(t_*)>0$ and $x^{in}(t_*)>0$. 
\end{defi}
 
The original chemostat models assume instant  growth after consumption ($\tau =0$), but for several biological systems this assumption seems unnatural, see 
\cite[Chapter 1]{smith2011introduction} for instance. Referring to experimental works to motivate the introduction of time delay, we remark \cite{collos1986time} where  the author reviewed several studies on the delayed response in the  growth of algae.  More recently,  Ellermeyer \emph{et al.} in \cite{ellermeyer2003theoretical} studied \emph{Escherichia coli}  assuming $\tau=20$ minutes.

Furthermore, for  biological models it is always fundamental to find criteria for the persistence of the organisms over an indefinitely long period of time (see \cite{waltman1991brief} for instance). This is the main concern  of this article.

\begin{defi}[Persistence]
System \ref{sistemaprincipal} is said to be \emph{persistent} if  
$$\liminf_{t\to\infty}x(t)>0$$
for every trajectory with  non-null initial conditions.
\end{defi}

\subsection{Persistence criteria}

Firstly, we observe that if the integral of the  function   $D$ is not infinity then     System \ref{sistemaprincipal}  is \emph{always} persistent. To see this  consider a solution with non-null initial condition and fix $t_0$ such that $x(t_0)>0$. Then  
\begin{align*}
  \frac{d}{dt}\left(x(t)e^{ \int_{t_0}^tD(r)\,dr}\right)&= \left(x'(t)+D(t)x(t)\right) e^{ \int_{t_0}^{t }D(r)\,dr}\\
  &=  x(t-\tau)p(s(t-\tau)) e^{ \int_{t_0}^{t-\tau}D(r)\,dr} \\
  &\geq 0,  
\end{align*}
which means that
$$x(t)\geq x(t_0)e^{-\int_{t_0}^tD(r)\,dr}\geq x(t_0)e^{-\int_{t_0}^\infty D(r)\,dr}>0 $$
for all $t\geq t_0$ and the system is persistent. Therefore to study persistence criteria  it seems natural to assume the function $D$ is \emph{large enough} such that its integral diverges.

In order to state Theorem \ref{persistencia} we need to introduce some functions. Throughout  this article we  repeatedly denote by $z$ a solution of the equation 
\begin{equation}\label{derivadaz}   
  z'(t) =\left( \si(t)-z(t)\right)D(t)    
\end{equation}
for $t\geq 0$, with $z(0)> 0$. As mentioned, this is the solution of the nutrient in System \ref{sistemaprincipal} in absence of organisms and  is known as the \emph{washout solution}. Note that $z$ is given by
\begin{equation}\label{formaz}
  z(t) =z(0)e^{-\int_0^tD(r)\,dr}+\int_0^t\si(h) D(h)e^{-\int_h^tD(r)\, dr}\, dh.
\end{equation}

We also work with 
\begin{equation}\label{funcionc}
  c(t):=c(0)e^{-\int_0^tD(r)\,dr}+\int_{-\tau}^{t-\tau}c(h)p(z(h))e^{-\int_h^tD(r)\,dr}\,dh,\, t\geq 0
\end{equation} 
and we assume that $c\geq 0$ in $[-\tau,0]$ with $c(0)>0$. Observe that $c$ is solution of the linear equation 
$$c'(t)=-D(t)c(t)+c(t-\tau)p(z(t-\tau))e^{-\int_{t-\tau}^tD(r)\,dr}.$$
Since $c(t)>0$ for all $t\geq 0$ we can define the function 
\begin{equation}\label{funcionphi}
  \varphi(t):=\frac{c(t)}{c(t+\tau)}e^{-\int_t^{t+\tau}D(r)\,dr},\quad t\geq 0 
\end{equation}
which is  explicitly given by 
\begin{equation*} 
  \varphi(t) = \frac{c(0)e^{-\int_0^tD(r)\,dr}+\int_{-\tau}^{t-\tau}c(h)p(z(h))e^{-\int_h^tD(r)\,dr}\,dh}{c(0)e^{-\int_0^tD(r)\,dr}+\int_{-\tau}^{t}c(h)p(z(h))e^{-\int_h^tD(r)\,dr}\,dh}.
\end{equation*}
Observe that multiples of $c$ result in the same $\varphi$ and that the image of $\varphi$ is contained in $(0,1]$. Furthermore, the following result,  which provides sufficient and necessary conditions for persistence, shows that this function $\varphi$ is \emph{inherent} to System \ref{sistemaprincipal} in the sence that it relevance is independet on the choince of $z(0)$ or $c$ over $[-\tau,0]$. 

Therefore, we can state the main result of this article.


\begin{thm}[Persistence criteria]\label{persistencia} 
Let $\tau$ be any non-negative constant and assume the function  $\si$ is upper and lower  bounded by positive constants, $D$ is upper   bounded by a positive constant, and  the integral of $D$ diverges. Furthermore, let $z$ be solution of \eqref{derivadaz} and  $\varphi$ defined in   \eqref{funcionphi}. 

Therefore  System  \ref{sistemaprincipal} is persistent if and only if there are positive constants   $\eta$ and $T$ such that 
\begin{equation}\label{condicion}
  \int_{t_1}^{t_2} p(z(t-\tau)) \varphi(t-\tau)\, dt> \int_{t_1}^{t_2}\left( D(t)+\eta\right)  \, dt
\end{equation}
for all $t_1> T$, $t_2-t_1> T$.
\end{thm}


\section{Particular cases and previous results}\label{seccioncasos}

\subsection{Periodic nutrient input and  washout} Let $\omega$ be a positive constant, we now study the particular case when $\si$ and $D$ are $\omega$-periodic. For an $\omega$-periodic function $f$ we denote its average as 
$$\langle f\rangle:=\frac{1}{\omega}\int_0^\omega f(t)\, dt.$$

Then we have the following result.

\begin{thm}[Persistence for periodic data]\label{persistenciaperiodica} 
Let $\tau$ be any non-negative constant and assume the functions  $\si$ and  $D$ are $\omega$-periodic with the former positive and the latter non-null. 

Then there are a unique  $\omega$-periodic function $z$ solution of \eqref{derivadaz} and a unique $c$ (up to a constant factor) such that $\varphi$ defined  in   \eqref{funcionphi} is $\omega$-periodic. Therefore  System  \ref{sistemaprincipal} is persistent if and only if  
\begin{equation*} 
   \langle p(z) \varphi \rangle>\langle  D\rangle.
\end{equation*}
\end{thm}

\subsection{Constant nutrient input and constant washout} This particular case serves to illustrate the ideas behind the proofs of Theorem \ref{persistencia} and Theorem \ref{persistenciaperiodica} and their relation to previous work. Assume that $\si$ and $D$ are constants and that $\tau>0$.   Ellermeyer in \cite[Theorem 3.3 \& Theorem 3.4]{ellermeyer1994competition}  proved that System \ref{sistemaprincipal} is persistent if and only if 
\begin{equation}\label{Dchico}
  p(\si)e^{-D\tau}>D.
\end{equation}
We get the same result in the following corollary of Theorem \ref{persistenciaperiodica}.
\begin{cor}
If $\si$ and $D$ are constants and  $\tau>0$ then System \ref{sistemaprincipal} is persistent if and only if inequality \eqref{Dchico} holds. 
\end{cor}

\begin{proof}
Define $z=\si$ and $\varphi\in(0,1)$  as the unique number that satisfies
$$\varphi =e^{-p(\si)\tau\varphi}.$$ 
We claim that 
\begin{equation}\label{Dchico2}
  p(\si)\varphi>D
\end{equation} 
is equivalent to \eqref{Dchico}. To see this assume \eqref{Dchico2}, then 
$$D<p(\si)\varphi=p(\si)e^{-p(\si)\varphi\tau} <p(\si)e^{-D\tau},$$
and \eqref{Dchico} holds. Reversing inequality,  \eqref{Dchico} implies \eqref{Dchico2}. 

Now define
$$c(t)=e^{t\left(-D+ p\left(\si\right)\varphi\right)} $$
and note that satisfies \eqref{funcionc}. Since
$$\frac{c(t)}{c(t+\tau)}e^{-\tau D}=e^{-\tau p\left(\si\right)\varphi}=\varphi$$
we get that  $\varphi$ is the unique function given by  Theorem \ref{persistenciaperiodica}. Finally, this establishes that System \ref{sistemaprincipal} is persistent if and only if \eqref{Dchico2} holds and the result is proved.
\end{proof}

Since the above corollary does not require the use of  function $\varphi$, it seems natural to ask whether this equivalence can be extended to a more general case: is it possible to replace $\varphi$ in inequality  \eqref{condicion} in Theorem \ref{persistencia}   by $\exp(-\int_{t-\tau}^tD(r)\,dr)$? The answer is negative, as we show in the following subsection.

\subsection{Constant nutrient input and variable washout}

We now prove that in the particular case of constant nutrient input,  the persistence of System \ref{sistemaprincipal} implies a generalization of inequality \eqref{Dchico}. Furthermore, with a counterexample, we show  that this  generalization does not imply the persistence of  System \ref{sistemaprincipal}.

\begin{prop}\label{implicacion}
Suppose the function $\si$ is positive constant and $D$ is upper bounded. Therefore, if  System \ref{sistemaprincipal} is persistent then there are positive constants $\eta$ and $T$ such that
\begin{equation}\label{condicionperiodicidad}
  \int_{t_1}^{t_2} p(z)e^{-\int_{t-\tau}^tD(r)\,dr}\,dt>\int_{t_1}^{t_2}\left(D(t)+\eta\right) \,dt
\end{equation} 
for all $t_1>T$ and $t_2-t_1>T$.
\end{prop}

The proof of this result is given in Section \ref{otraspruebas}. In the following example, we show that the condition  \eqref{condicionperiodicidad} is   not sufficient to ensure persistence.

\begin{ex}
Consider the case $p(z)=\pi$, $\tau=\pi/2$, $D(t)=1-\sin(t)$, $ \eta\in(0,1/100)$ and  $  T=300\pi$. Then 
\begin{equation}\label{condicion1ejemplo}
  \int_{t_1}^{t_2} p(z )e^{-\int_{t-\tau}^tD(r)\, dr}\,dt>\int_{t_1}^{t_2}\left(D(t)+  \eta\right)\, dt
\end{equation} 
for all $t_1> T$ and $t_2-t_1>  T$, and yet System \ref{sistemaprincipal} is not persistent. 
\end{ex}
To show this note that for all $h\in\R$ we get 
\begin{align*}
  \int_h^{2\pi+h}  p(z)e^{-\int_{t-\tau}^tD(r)\, dr } \, dt&=\int_h^{2\pi+h} \pi e^{-\pi/2+\sin(t)-\cos(t)}\, dt\\
  &=\int_0^{2\pi} \pi e^{-\pi/2+\sin(t)-\cos(t)}\, dt\\
  &>6.42.
\end{align*}
Let be $n\in\N$ such that $2\pi n\leq t_2-t_1<2\pi(n+1)$. Then $n\geq 150$,  we obtain
\begin{align*}
   \int_{t_1}^{t_2}\left(D(t)+  \eta\right)\, dt&=(1+ \eta)(t_2-t_1)+\cos(t_2)-\cos(t_1)\\
   &<  (t_2-t_1)1.01+2\\
   &< 2\pi(n+1)1.01+2\\
   &<6.42n\\
   &\leq   \int_{t_1}^{t_2} p(z )e^{-\int_{t-\tau}^tD(r)\, dr}\,dt, 
\end{align*}
and  condition \eqref{condicion1ejemplo} is satisfied. 

We now focus on persistence. Consider the constant $\varphi$ such that 
$$\varphi=e^{-\int_{t-\pi/2}^t\pi\varphi dr}=e^{-\varphi\pi^2/2},$$
 then  $\varphi<0.3.$ Furthermore, both $p(z)\varphi$ and $D$ are $2\pi$-periodic functions that satisfy 
$$\langle \pi\varphi\rangle =\varphi 2\pi^2<2\pi=\langle D\rangle.$$
Therefore, by Theorem \ref{persistenciaperiodica}  System \ref{sistemaprincipal} is not persistent. This ends the example.

\begin{rem}
A remaining question in this article is whether Proposition \ref{implicacion} is valid for a general nutrient input function.
\end{rem}

\subsection{Instantaneous response in growth}

Now assume  $\tau=0$ in System \ref{sistemaprincipal}. Therefore we get the following result.

\begin{cor}\label{corolariosindelay}
Let  $\tau= 0$ and assume the functions   $\si$ and $D$ are  upper bounded.  Then  System   \ref{sistemaprincipal} is persistent if and only if there are positive constants $\eta$ and $T$ such that 
$$\int_{t_1}^{t_2} p(z(t ))  \, dt> \int_{t_1}^{t_2}\left(D(t)+\eta\right) \, dt$$
for all $t_1> T$, $t_2-t_1> T$.
\end{cor}

The proof follows  directly from Theorem \ref{persistencia} since $\varphi=1$. Furthermore, the assumption  that $\si$ is lower bounded by a positive constant is not required.  This is because to prove Theorem \ref{persistencia} this assumption is only necessary in the case that $\tau>0$ (see proof of Lemma \ref{lemaprincipal}).

Ellermeyer \emph{et al.} studied System \ref{sistemaprincipal} with instant  response in growth  and with constant washout  for nutrient and (different) constant washout for organisms in \cite{ellermeyer2001persistence}.  Therefore, Corollary \ref{corolariosindelay} agrees with  Theorem 3 therein   if we assume $D$ constant in the former and   the same washout for $s$ and $x$ in the  latter.

\section{Proof of the main result}\label{pruebas}

To begin with the study of System \ref{sistemaprincipal} we first analyze the existence and uniqueness of solutions. Furthermore, as mentioned, for physics reasons $s$ and $x$ must be non-negative functions. Since $p(s)$ is locally Lipschitz continuous we get the local existence and uniqueness of solutions. Now fix $\varepsilon>0$ and $T\in(0,\varepsilon)$ such that a solution  exists on $[-\tau,\varepsilon)$. Then $s$ and $x$ are bounded on $[-\tau,T]$ and so on $p'(s)$. By the mean value theorem applied to $p$, we get 
$$s'(t)\geq -D(t)s(t)-p(s)x(t)\geq-Ms(t)$$
for some constant $M=M(T)$,  which implies $s\geq 0$ on $[0,T]$.  Similarly $x$ is non-negative on $[0,T]$. Taking limit on $T$ ensures the functions $s$ and $x$ are non-negative over $[-\tau,\varepsilon)$. Therefore
\begin{align*}
  s'(t)&\leq \max_{h\in[0,\varepsilon]}D(h)\si(h),\\
  x'(t)&\leq \max_{h\in[-\tau,\varepsilon-\tau]} x(h )p(s(h ))e^{-\int_{h }^{h+\tau}D(r)\, dr},
\end{align*}  
which provides the solution can be extended over $[-\tau,\infty)$. Finally, note that if $x(t)>0$ for some $t\geq 0$ then $x(h)>0$ for all $h\geq t$. Therefore we have the following observation. 

\begin{rem}\label{observaciondefinicion}
System \ref{sistemaprincipal}   is persistent if and only if for every solution $(s,x)$ with non-null initial conditions  there exists   $\delta>0$ such that  $x(t)\geq \delta$ for all $t\geq\tau$.
\end{rem}

\subsection{Main ideas of the proof}\label{seccionideas}

We sketch the plan for proving  Theorem \ref{persistencia}. Assume   $x(t)$ is positive for   $t\geq \tau$, therefore 
\begin{equation}\label{comportamientox}
\begin{aligned}
  \frac{d}{dt}\ln(x(t))&=-D(t)+\frac{x(t-\tau)}{x(t)}p(s(t-\tau))e^{-\int_{t-\tau}^tD(r)\, dr} \\
  &=-D(t)+p(s(t-\tau))\psi(t-\tau)
\end{aligned}
\end{equation}
if we define
\begin{equation}\label{funcionpsi}
  \psi(t):=\frac{x(t)}{x(t+\tau)}e^{-\int^{t+\tau}_tD(r)\, dr}
\end{equation}
for all $t\geq 0$. This equality is the reason to define $\varphi$ as in \eqref{funcionphi}. Note that 
\begin{equation}\label{formax}
  x(t+\tau )=x(t_0)e^{\int_{t_0}^{t+\tau}p(s(h-\tau))\psi(h-\tau)-D(h)\, dh} 
\end{equation}
for $t+\tau\geq t_0\geq \tau $ and taking $t_0=t$ we obtain 
\begin{equation}\label{igualdadpsi}
   e^{-\int^{t+\tau}_tp(s(h-\tau))\psi(h-\tau)\, dh}= e^{-\int^{t }_{t-\tau}p(s(h))\psi(h)\, dh}=\psi(t) 
\end{equation}
for all $t\geq  \tau$. 

Now assume  System \ref{sistemaprincipal} is persistent. In the proof of the main theorem, we show that this implies that for large times it happens that  $z\gg s $. Furthermore, this implies $p(z)\gg p(s)$ and there is $\eta$ positive such that
$$\int p(z(t ))\varphi(t)\, dt\geq \int \left( p(s(t ))\psi(t)+2\eta\right)\, dt $$
(Lemma \ref{vuelta}). Then $x$ and $\ln(x)$ are both bounded (Lemma \ref{funcionz}) and using \eqref{comportamientox} we can find $t_2$ larger enough than $t_1$ such that
$$\int_{t_1}^{t_2}\left( -D(t)+p(s(t -\tau))\psi(t-\tau)+ \eta\right)\, dt=\ln(x(t_2))-\ln(x(t_1))+(t_2-t_1)\eta\geq 0 .$$
Finally, we get 
$$\int_{t_1}^{t_2}  p(z(t-\tau)) \varphi(t-\tau) \, dt> \int_{t_1}^{t_2} \left(D(t)+\eta \right) \, dt.$$
These are the main ideas behind the proof of the first half of Theorem \ref{persistencia}. 

On the other hand, assume \eqref{condicion}. To prove persistence we reason by contradiction. If  $x$ is  small enough, then  it happens that $z\approx s$.  This implies that   $\psi \approx \varphi $ (Lemma \ref{lemaprincipal}), $p(z)\approx p(s)$,  and
$$\int \left(p(s(t-\tau))\psi(t-\tau)-D(t)\right)\, dt\approx\int\left( p(z(t-\tau))\varphi(t-\tau)-D(t)\right)\, dt >0.$$
Returning to \eqref{formax} we see that $x$ cannot decrease to zero.

\subsection{Preliminary results}

We state and prove several  technical lemmas  for the proof of Theorem \ref{persistencia}. From now on and until the end of this section we assume the hypothesis of this theorem: $\si$ is upper and lower bounded by positive constants, and $D$ is upper bounded and its integral diverges. Also we introduce functions $f,\,g:[-\tau,\infty)\to\R^+$ both upper bounded by a positive constant $M$ and $f$ also lower bounded by a positive constant.  Furthermore, we ask $M$ satisfies 
\begin{equation}\label{cotaM}
  M\geq \frac{1}{4\tau+1}.
\end{equation} 

As it was done in \cite{ellermeyer1994competition}, we work with the concentration of nutrient stored internally given by
\begin{equation}\label{funciony}
  y(t):=\int_{t-\tau}^t x(h)p(s(h))e^{-\int_h^tD(r)\,dr}\,dh.
\end{equation}

\begin{lem}\label{funcionz} 
The functions  $z,\,s,\,x$ and  $y$ given in \eqref{derivadaz}, \ref{sistemaprincipal} and \eqref{funciony}, respectively, satisfy
$$|z(t)-s(t)-x(t)-y(t)|\to 0$$
when $t$ goes to infinity. Furthermore, they all are upper bounded.
\end{lem}

\begin{proof} 
To begin note that
\begin{align*}
  \frac{d}{dt}y(t)&=x(t)p(s(t))-  x(t-\tau)p(s(t-\tau))e^{ -\int_{t-\tau}^t D(r)\,dr}-D(t)y(t).
\end{align*}
Then 
\begin{align*}
  \frac{d}{dt}(z-s-x-y) &=\si(t)D(t)-D(t)z(t)\\
  &\quad -(\si(t)-s(t))D(t)+x(t)p(s(t))\\
  &\quad -x(t-\tau)p(s(t-\tau))e^{-\int_t^\tau D(r)\, dr}+x(t)D(t)\\
  &\quad -x(t)p(s(t))+  x(t-\tau)p(s(t-\tau))e^{ -\int_{t-\tau}^t D(r)\,dr}+y(t)D(t)\\
  &=-D(t)(z(t)-s(t)-x(t)-y(t))
\end{align*}
and therefore
$$\left|z(t)-s(t)-x(t)-y(t)\right| =|z(0)-s(0)-x(0)-y(0)|e^{-\int_0^tD(r)\,dr}.$$ 
This proves the first statement of the lemma. Finally, by equality \eqref{formaz},  
\begin{align*}
  z(t)&= z(0) e^{-\int_0^tD(r)\, dr}+\int_0^t \si(h) D(h)e^{-\int_h^tD(r)\, dr}\, dh \\
  &\leq  \max_{h\geq 0}\left\{z(0),\si(h)\right\}  \left( e^{-\int_0^tD(r)\, dr}+\int_0^t  D(h)e^{-\int_h^tD(r)\, dr}\, dh\right),\\
  &= \max_{h\geq 0}\left\{z(0),\si(h)\right\} 
\end{align*}
which implies that $z$ is upper bounded. Since   $z,\,s,\,x$ e $y$ are non-negative  functions, by  the previous statement, they all are upper bounded. 
\end{proof}

\begin{lem}\label{propiedadphi}
 The function $\varphi$ defined in \eqref{funcionphi} satisfies
$$\varphi(t)=e^{-\int_{t-\tau}^t\varphi(h)p(z(h))\,dh}.$$
\end{lem}

\begin{proof}
To see this note that $c(t)$ defined in \eqref{funcionc} is positive for all $t\geq 0$ and satisfies
\begin{align*}
  c'(t)&=c(t)\left(-D(t)+\frac{c(t-\tau)}{c(t)}p(z(t-\tau))e^{-\int_{t-\tau}^tD(r)\,dr}\right)\\
  &= c(t)\left(-D(t)+\varphi(t-\tau)p(z(t-\tau)) \right).
\end{align*}
Then 
\begin{align*}
  c(t+\tau)=c(t)e^{\int_t^{t+\tau}\left(-D(h)+\varphi(h-\tau)p(z(h-\tau)) \right)\,dh}
\end{align*}
and we finally get
$$e^{\int_t^{t+\tau} \varphi(h-\tau)p(z(h-\tau))  \,dh} =\frac{c(t)}{c(t+\tau)}e^{-\int_t^{t+\tau}  D(h) \,dh}=\varphi(t).$$
 \end{proof}

The next lemma is the key to proving the first half of Theorem \ref{persistencia}. 

\begin{lem}\label{vuelta}
Let $t_0\geq 0$ and $\tau>0$. Consider $\varphi,\,\psi:[t_0-\tau,\infty)\to(0,1]$ satisfying 
\begin{align*}
    \varphi(t)&=e^{-\int^t_{t-\tau}f(h)\varphi(h)\, dh}, \\
    \psi(t)&=e^{-\int^t_{t-\tau}g(h)\psi(h)\, dh}  
\end{align*}
for all $t> t_0$ and assume there is    $\varepsilon>0$   such that  
\begin{equation}\label{distancia}
  f(t)> g(t)+\varepsilon
\end{equation}
for all  $t\geq t_0$. Then there are positive constants $\eta$ and $T$ such that  
$$\int_{t_1 }^{t_2}f(t-\tau)\varphi(t-\tau)\, dt\geq  \int_{t_1 }^{t_2} \left(g(t-\tau)\psi(t-\tau)+2\eta \right)\, dt$$
for all  $t_2-T\geq t_1\geq t_0$.
\end{lem}

\begin{proof}
Firstly, define 
\begin{align*}
\alpha &=1+ \varepsilon/(2M)\\
  \eta &=\min\left\{ \ln\left(\alpha \right)/(3\tau), \varepsilon e^{-\tau M}/6 \right\},\\
  T&= M\tau /\eta+3\tau.
\end{align*}
and fix $t_1$ and $t_2$ satisfying $t_2-T\geq t_1\geq t_0$. We sketch the   proof in three steps. 
 
\emph{Step 1.} Observe that if  $\alpha \varphi(  t )\leq  \psi(  t ) $ for $t\geq t_0$, then we get   
\begin{align*}
  e^{\int_{t -\tau}^{t} \left(f(h)\varphi(h)-g(h)\psi(h)\right) \, dh}&=\frac{\psi(t)}{\varphi(t)}\geq  \alpha 
\end{align*}
which implies
$$\int_{t-\tau}^{t} \left(f(h)\varphi(h)-g(h)\psi(h) \right)\, dh\geq \int_{t-\tau}^{t}\frac{\ln\left(\alpha\right)}{\tau}\, dh\geq \int_{t-\tau}^{t}3\eta\, dh.$$

On the other hand, if there are $h_2\geq h_1\geq t_0$ such that  $ \alpha \varphi(  t )\geq \psi(  t ) $ for all $t\in [h_1,h_2]$ then  
\begin{equation*}
  \begin{aligned}
  \left(g(t)+\varepsilon/2\right)\varphi(t)-g(t)\psi(t)&\geq \left(g(t)+\varepsilon/2\right)\varphi(t)-g(t) \alpha \varphi(t )\\
  &\geq \left(g(t)+\frac{\varepsilon g(t)}{2M}\right)\varphi(t)-g(t) \alpha \varphi(t )\\
  &= 0.
  \end{aligned}
\end{equation*} 
By  \eqref{distancia} and since  $\varphi(t) \geq e^{-\tau M}$ for all $t\geq t_0$,   we get
\begin{align*}
  \int_{h_1}^{h_2} \left(f(t)\varphi(t)-g(t)\psi(t) \right)\, dt&\geq   \int_{h_1}^{h_2}\left( (g(t)+\varepsilon) \varphi(t)-g(t)\psi(t) \right)\, dt\\
  & =\int_{h_1}^{h_2} \left(\left(g(t)+\frac{\varepsilon}{2}\right) \varphi(t)-g(t)\psi(t)+ \frac{\varepsilon}{2} \varphi(t ) \right) \, dt\\
  &\geq \int_{h_1}^{h_2}  \frac{\varepsilon e^{-M\tau}}{2} \, dt\\
  &\geq \int_{h_1}^{h_2} 3 \eta \, dt.
\end{align*}

\emph{Step 2.} We now claim there is a finite decreasing sequence $(h_n)_{1\leq n\leq N}\subset\R$ with the properties that $\alpha\varphi( h_n )\leq  \psi(  h_n)$ for all $1\leq n\leq N-1$ and that $\alpha\varphi( t )\geq  \psi(  t)$ if  
$$t\in I=\bigcup_{n=1}^{N-1}[h_{n+1},h_n-\tau]\cup [h_1,t_2-\tau] .$$
Note that  the step above  implies
\begin{equation}\label{caminoenreverso}
  \begin{aligned}
  \int_{h_n-\tau }^{h_n}f(h)\varphi(h)\, dh&\geq  \int_{h_n-\tau }^{h_n}\left( g(h)\psi(h)+3\eta \right)\, dh,\\
  \int_If(h)\varphi(h)\, dh&\geq  \int_I\left( g(h)\psi(h)+3\eta\right) \, dh 
  \end{aligned}
\end{equation}
for all $1\leq n\leq N-1$. Therefore, define  
\begin{equation*}
  h_1 =\left\{\begin{aligned}
    &t_2-\tau &\text{if }  \alpha \varphi( t_2-\tau )< \psi(  t_2-\tau),\\
   &\inf\left\{\begin{array}{l}
   t\geq   t_1-\tau : \alpha \varphi(  h )\geq \psi(  h )\\
   \text{for all }h\in[t,t_2-\tau]
   \end{array}\right\} & \text{otherwise.}
  \end{aligned}\right.
\end{equation*}
For $n\geq 1$ and while $t_1\leq h_{n } $, define
\begin{equation*}
  h_{n+1} =\left\{\begin{aligned}
    &h_{n }-\tau& \text{if } \alpha \varphi( h_{n }-\tau)< \psi( h_{n }-\tau),\\
   &\inf\left\{\begin{array}{l}
   t\geq   t_1 -\tau: \alpha \varphi(  h )\geq \psi(  h )\\
   \text{for all }h\in[t,h_n-\tau]
   \end{array}\right\} &\text{otherwise.}
  \end{aligned}\right.
\end{equation*}
Observe that the sequence ends when $t_1> h_{N } \geq t_1-\tau$. Taking into consideration \eqref{caminoenreverso},  we get
$$\int_{h_N }^{t_2-\tau}f(h)\varphi(h)\, dh\geq  \int_{h_N }^{t_2-\tau}\left( g(h)\psi(h)+3\eta \right)\, dh.$$
 
\emph{Step 3.} Finally, notice that
\begin{align*}
  \int_{t_1-\tau}^{t_2-\tau}f(t)\varphi(t)\, dt&\geq \int_{h_N}^{t_2}f(t)\varphi(t)\, dt\\
  &\geq \int_{h_N }^{t_2-\tau} \left(g(t)\psi(t)+2\eta \right)\, dt+(t_2-\tau-h_N)\eta\\
  &\geq \int_{t_1 -\tau}^{t_2-\tau} \left(g(t)\psi(t)+2\eta \right)\, dt+(T-\tau)\eta\\
  &\quad -\int_{t_1-\tau }^{h_N} \left(g(t)\psi(t)+2\eta \right)\, dt \\
  &\geq \int_{t_1-\tau }^{t_2-\tau}\left( g(t)\psi(t)+2\eta \right)\, dt+(T-\tau)\eta-\tau (M+2\eta) \\
  &= \int_{t_1 -\tau}^{t_2-\tau} \left(g(t)\psi(t)+2\eta \right)\, dt+(T-M\tau/\eta-3\tau)\eta \\
  &= \int_{t_1 -\tau}^{t_2-\tau} \left(g(t)\psi(t)+2\eta \right)\, dt 
\end{align*}
by definition of $T$, and the lemma is proved.
\end{proof}

We now continue with the   lemmas that we need to prove the second half of Theorem \ref{persistencia}.

\begin{lem}\label{hermana}
Let $0\leq t_0<t_1$, $\tau>0$  and $\varphi:[t_0-\tau,\infty)\to(0,1]$ be such that 
$$\varphi(t)=e^{-\int_{t-\tau}^tf(h)\varphi(h)\, dh}$$
for all $t>t_0.$
Assume that there is $\varepsilon>0$ such that   $|f(t)-g(t)|<\varepsilon$ for all $t\in [t_0,t_1]$. Then  there exists  $\tilde\psi:[t_0-\tau,\infty)\to(0,1]$ that satisfies 
$$\tilde\psi(t)=e^{-\int_{t-\tau}^tg(h)\tilde\psi(h)\, dh} $$
for all $t>t_0$ and such that 
$$|\varphi(t)-\tilde\psi(t)|<2\varepsilon\tau e^{2M(t-t_0)} $$
for all $t\in [t_0-\tau,t_1]$.
\end{lem}

\begin{proof}
Firstly need to define a function $\tilde\psi:[t_0-\tau ,\infty)\to (0,1]$ satisfying 
$$\left\{\begin{aligned}
\tilde\psi(t)&=\varphi(t)&\text{ if }t\in[t_0-\tau,t_0 ],\\
\tilde\psi(t)&=e^{-\int_{t-\tau}^tg (h)\tilde\psi(h)\,dh}&\text{ if }t>t_0 .
\end{aligned}
\right.
$$
For this purpose consider $d$ solution of
\begin{align*}
  d'(t)&=\left\{\begin{array}{ll}
  d(t)\left(-D(t) + \varphi(t-\tau )g(t-\tau)\right) &\text{ if }t\in[t_0,t_0+\tau],\\
  -D(t)d(t) +d(t-\tau)g(t-\tau)e^{-\int_{t-\tau}^tD(h)\, dh}&\text{ if }t\in(t_0+\tau,\infty). 
\end{array}\right.
\end{align*}
with $d(t_0)>0$. Reasoning as in  the begining of Section \ref{pruebas}, $d$ is well defined and positive. Therefore we define
$$\tilde\psi(t)= \left\{\begin{array}{ll}
 \varphi(t)&\text{ if }t\in[t_0-\tau,t_0 ],\\
 \frac{d(t)}{d(t+\tau)}e^{-\int_{t-\tau}^tD(r)\,dr}&\text{ if }t>t_0 .
\end{array}
\right.
$$
 As it was done in \eqref{comportamientox}, for all $t\geq t_0$, we get
$$d(t+\tau)=d(t) e^{\int_t^{t+\tau} \left(\tilde\psi(h-\tau)g(h-\tau)-D(h)\right)\,dh}$$
and then
$$\tilde\psi(t)= \frac{d(t)}{d(t+\tau)} e^{\int_t^{t+\tau}  -D(h)\,dh} = e^{-\int_{t-\tau}^{t } \tilde\psi(h )g(h ) \,dh} $$
as we need. 

To continue with the proof define 
$$\mathcal{S}=\left\{t\geq t_0-\tau:|\varphi(h)-\tilde\psi(h)|<2\varepsilon\tau e^{2M(h-t_0)}\text{ for all }h\in[t_0-\tau,t]\right\}.$$
Observe that $\mathcal{S}$ is not empty since   $ t_0 \in\mathcal{S}$. Also note that $\varphi$ and $\tilde\psi$ might not be continuous at $t_0$ but they are for $t>t_0$. For this reason we first claim that if  $\delta=\min\{\varepsilon\tau/M,\tau\}$ 
then  $t_0 +\delta\in\mathcal{S}$. To see this  note that by the mean value theorem, for positive  $a$ and $b$, we get $ |e^{-a}-e^{-b}|\leq |a-b| $. Then,  for all $t\in(t_0 ,t_0 +\delta]$, it happens
\begin{align*}
  |\varphi(t)-\tilde\psi(t)|&=\left|e^{-\int_{t -\tau}^{t }f(h)\varphi(h)\,dh}-e^{-\int_{t -\tau}^{t }g(h)\tilde\psi(h)\,dh}\right|\\
  &\leq \left|\int_{t -\tau}^{t_0 }f(h)\varphi(h)\,dh+\int_{t_0 }^{t}f(h)\varphi(h)\,dh\right.\\
  &\hspace{2cm}\left.-\int_{t -\tau}^{t_0 }g(h)\varphi(h)\,dh-\int_{t_0 }^{t}g(h)\tilde\psi(h)\,dh\right|\\
  &=\int_{t -\tau}^{t_0 }|f(h)-g(h)|  \varphi(h)\,dh+\int_{t_0 }^{t}|f(h)\varphi(h)-g(h)\tilde\psi(h)|\,dh\\
  &\leq \varepsilon(t_0 -(t-\tau))+M(t- t_0  )\\
  &\leq \varepsilon\tau+M\delta\\
  &<2\varepsilon\tau e^{2M(t-t_0)}.
\end{align*}

We now need to show that   $t^*=\sup\mathcal{S}\geq t_1$. By contradiction assume that  $t^*\in[t_0 +\delta,t_1)$. Since $t^*\geq t_0+\delta$, by continuity it must be 
\begin{equation}\label{contradiccionlemaexistencia}
  |\varphi(t^*)-\tilde\psi(t^*)|=2\varepsilon\tau e^{2M(t^*-t_0)} .
\end{equation}
But, on the other hand,
\begin{align*}
  |\varphi(t^*)-\tilde\psi(t^*)|&=\left|e^{-\int_{t^*-\tau}^{t^*}f(h)\varphi(h)\,dh}-e^{-\int_{t^*-\tau}^{t^*}g(h)\tilde\psi(h)\,dh}\right|\\
  &\leq \left| \int_{t^*-\tau}^{t^*}f(h)\varphi(h)\,dh - \int_{t^*-\tau}^{t^*}g(h)\tilde\psi(h)\,dh \right|\\
  &\leq  \int_{t^*-\tau}^{t^*}\left(|g(h)\varphi(h)-g(h)\tilde\psi(h)|+\varepsilon \varphi(h)\right)\, dh\\
  &\leq M\int_{t^*-\tau}^{t^*}| \varphi(h)- \tilde\psi(h)|\, dh+\varepsilon\tau\\
  &\leq M\int_{t^*-\tau}^{t^*} 2\varepsilon\tau e^{2M(h-t_0)}\, dh+\varepsilon\tau\\
  &=\varepsilon\tau\left(e^{2M(h-t_0)}\Big|_{t^*-\tau}^{t^*}+1\right)\\
  &=\varepsilon\tau\left(e^{2M(t^*-t_0)}(1-e^{-2M\tau})+1\right)\\
  &<2\varepsilon\tau e^{2M(t^*-t_0)} .
\end{align*}
This is a contradiction to \eqref{contradiccionlemaexistencia} and the lemma is proved. 
\end{proof}

\begin{lem}\label{lemacrucial} 
Let  $t_0\geq 0$, $\tau> 0$ and consider functions $\varphi,\,\psi:[t_0-\tau,\infty)\to(0,1]$  satisfying 
\begin{align*}
  \varphi(t)&=e^{-\int^t_{t-\tau}f(h)\varphi(h)\, dh},\\
  \psi(t)&=e^{-\int^t_{t-\tau}f(h)\psi(h)\, dh}
\end{align*}
for all $t\geq t_0$. Therefore  
$$|\varphi(t)-\psi(t)|< 3M \sqrt{\frac{ (t-t_0)}{ \inf f}} \left(1-e^{-M\tau}\right)^{(t-t_0)/(2\tau)-1/2} $$
for all $t\geq t_0 $  where $\inf f$ means the infimum of $f$. 
\end{lem}

\begin{proof}
Define
$$\begin{aligned}
m(t)&=\varphi(t)-\psi(t),\\
\beta&=1-e^{-M\tau} ,
\end{aligned}
$$ 
and note that $m\in[-1,1]$. We sketch the proof in four steps, the first three dedicated to show that  
$$\int_{t}^{t +\tau}f(h)|m(h)|\, dh\leq M  (t-t_0) \beta^{(t-t_0)/\tau-1} $$
for all $t\geq t_0$ and the last to prove the lemma itself.

\emph{Step 1.} 
For all $t_*\geq t_0$  we claim that there is $t\in(t_*,t_*+\tau)$ such that $m(t)=0$. By contradiction assume that $m>0$ on $(t_*,t_*+\tau)$. Since $f>0$ we get  
$$\varphi(t_*+\tau)=e^{-\int_{t_*}^{t_*+\tau}f(h)\varphi(h)\, dh}<e^{-\int_{t_*}^{t_*+\tau}f(h)\psi(h)\, dh}=\psi(t_*+\tau)$$
and  $m(t_*+\tau)<0$, which is a contradiction because $m$ is continuous. And  similarly for the case  $m<0$ on $(t_*,t_*+\tau)$.

\emph{Step 2.} 
Let be $t_*\geq t_0$ such that   $m(t_*)=0$. We claim that  
\begin{equation}\label{funcionm}
  \int_{t_*}^tf(h)|m(h)|\, dh\leq \beta \int_{t_*-\tau}^{t-\tau}f(h)|m(h)|\, dh
\end{equation}
for all $t\geq t_*$.   By the mean value theorem, note that  
\begin{equation*}
  \begin{aligned}
  m(t)&=e^{-\int^t_{t-\tau}f(h)\varphi(h)\, dh}-e^{-\int^t_{t-\tau}f(h)\psi(h)\, dh}\\
  &=-e^{-c}\left(\int_{t-\tau}^tf(h)\varphi(h)\, dh-\int_{t-\tau}^tf(h)\psi(h)\, dh\right)\\
  &=-e^{-c}\int_{t-\tau}^tf(h)m(h)\, dh
  \end{aligned}
\end{equation*}
where $c=c(t)>0$, but for simplicity we do not write the dependence on $t$. Therefore  
\begin{equation*}\label{igualacero}
  0= \int_{t_*-\tau}^{t_*}f(h)m(h)\, dh, 
\end{equation*}
which implies 
\begin{align*}
  \int_{t-\tau}^tf(h)m(h)\, dh&=\int_{t_*}^tf(h)m(h)\, dh+\int_{t_*-\tau}^{t_*}f(h)m(h)\, dh-\int_{t_*-\tau}^{t-\tau}f(h)m(h)\, dh\\
  &=\int_{t_*}^tf(h)m(h)\, dh  -\int_{t_*-\tau}^{t-\tau}f(h)m(h)\, dh
\end{align*}
for all $t\geq t_*$. Then 
\begin{align*}
  m(t)&= e^{-c} \left( \int_{t_*-\tau}^{t-\tau}f(h)m(h)\, dh-\int_{t_*}^tf(h)m(h)\, dh  \right).
\end{align*}
Suppose that  $m\geq 0$ on  $[t_*,t_1]$ for some $t_1\in(t_*,t_*+\tau]$, applying  Gr\"onwall's inequality, we get
\begin{align*}
  f(t)m(t)&\leq f(t)\left( \int_{t_*-\tau}^{t-\tau}f(h)m(h)\, dh-\int_{t_*}^tf(h)m(h)\, dh  \right)\\
  &\leq M\left( \int_{t_*-\tau}^{t-\tau}f(h)m(h)\, dh-\int_{t_*}^tf(h)m(h)\, dh  \right)\\
\end{align*}
for all  $t\in[t_*,t_1]$, and then
$$\frac{d}{dt}e^{M(t-t_*)}\int_{t_*}^tf(h)m(h)\, dh \leq e^{M(t-t_*)}M \int_{t_*-\tau}^{t-\tau}f(h)|m(h)|\, dh.$$
By Fubini–Tonelli theorem, we obtain
\begin{align*}
  \int_{t_*}^tf(h)|m(h)|\, dh  &=\int_{t_*}^tf(h) m(h) \, dh  \\
  & \leq e^{M(t-t_*)}\int_{t_*}^tf(h)m(h)\, dh \\
  &\leq \int_{t_*}^t e^{-M(t-r)}M \int_{t_*-\tau}^{r-\tau}f(h)|m(h)|\, dhdr\\
  &= \int_{t_*-\tau}^{t-\tau}f(h)|m(h)|  \int_{h+\tau}^t e^{-M(t-r)}M \,drdh\\
  &\leq \int_{t_*-\tau}^{t-\tau}f(h)|m(h)|\, dh\int_{t_*}^t e^{-M(t-r)}M \,dr\\
  &\leq \int_{t_*-\tau}^{t-\tau}f(h)|m(h)|\, dh \left(1-e^{-M(t-t_*)}\right)\\
  &\leq \beta  \int_{t_*-\tau}^{t-\tau}f(h)|m(h)|\, dh 
\end{align*}
since $0\leq t-t_*\leq \tau$. Then  
\begin{equation}\label{desigualdadtestrella}
  \int_{t_*}^tf(h)|m(h)|\, dh   \leq \beta  \int_{t_*-\tau}^{t-\tau}f(h)|m(h)|\, dh 
\end{equation} 
for all $t\in[t_*,t_1]$. Note that if $m\leq 0$ on  $[t_*,t_1]$, similarly, we get the same result.

Now fix any $t\geq t_*$. By Step 1 we can split the interval in a disjoint union
$$[t_*,t)=\bigsqcup _{n=0}^{N} [h_n,h_{n+1})$$
where $h_0=t_*$, $h_{N+1}=t$, $h_{n+1}\in (h_n,h_n+\tau)$, $m(h_n)=0$  and $m$ does not change its sign on each of all the intervals $[h_n,h_{n+1})$ for all $0\leq n\leq N$. Since \eqref{desigualdadtestrella} is also valid for all $h_n$ (with $n\leq N$) instead of $t_*$, we get
\begin{align*}
  \int_{t_*}^tf(h)|m(h)|\, dh&=\sum _{n=0}^{N}\int_{h_n}^{h_{n+1}}f(h)|m(h)|\, dh\\
  & \leq \beta \sum _{n=0}^{N}\int_{h_n-\tau}^{h_{n+1}-\tau}f(h)|m(h)|\, dh\\
   &=   \beta \int_{t_*-\tau}^{t-\tau}f(h)|m(h)|\, dh
\end{align*} 
and  \eqref{funcionm} is proved.

\emph{Step 3.} 
Now fix any $t\geq t_0$ and let   $N\in\N\cup\{0\}$ be such that 
$$N\tau\leq t-t_0<(N+1)\tau.$$
We claim that
$$\int_{t}^{t+\tau  }f(h)|m(h)|\, dh \leq \beta^N \int_{t-N\tau-\sum_{n=0}^N \delta_n}^{t-N\tau }f(h)|m(h)|\, dh $$
where $\delta_n\in[0,\tau)$ such that $m(t-n\tau-\delta_n)=0$ for all $0\leq n\leq N$. We proceed by induction. For the case $N=0$, by Step 1 there is $\delta_0\in[0,\tau)$  such that $m(t-\delta_0)=0$, and then by Step 2
\begin{align*}
  \int_{t}^{t+\tau  }f(h)|m(h)|\, dh&\leq \int_{t-\delta_0}^{t +\tau }f(h)|m(h)|\, dh\\
  &\leq \beta \int_{t-\delta_0 - \tau}^{t }f(h)|m(h)|\, dh  .  
\end{align*}
If we assume it is valid for $N-1$, again  by Step 1 there exists $\delta_N\in [0,\tau)$ such that 
$$m\left(t-(N-1)\tau-\sum_{n=0}^{N-1} \delta_n-\delta_N\right)=0,$$
and then
\begin{align*}
  \int_{t}^{t+\tau  }f(h)|m(h)|\, dh&\leq \beta^{N-1} \int_{t-(N-1)\tau-\sum_{n=0}^{N-1} \delta_n}^{t-(N-1)\tau }f(h)|m(h)|\, dh\\
  &\leq \beta^{N-1} \int_{t-(N-1)\tau-\sum_{n=0}^{N} \delta_n}^{t-(N-1)\tau }f(h)|m(h)|\, dh\\
  &\leq \beta^N \int_{t-N\tau-\sum_{n=0}^N \delta_n}^{t-N\tau }f(h)|m(h)|\, dh  .
\end{align*}

Finally, since $\sum_{n=0}^N \delta_n\leq N\tau\leq t-t_0$ and $ (t-t_0)/\tau-1<N$, we get
\begin{align*}
  \int_{t}^{t+\tau  }f(h)|m(h)|\, dh& \leq  \beta^N M  \sum_{n=0}^N \delta_n \\
  &\leq   M(t-t_0)\beta^{(t-t_0)/\tau-1}.
\end{align*}

\emph{Step 4.} Since
$$\varphi'(t)=\frac{d}{dt}e^{-\int_{t-\tau}^tf(h)\varphi(h)\, dh}=\varphi(t)[f(t-\tau)\varphi(t-\tau)-f(t)\varphi(t)],$$
and similarly, for  $\psi$, we get  
\begin{align*}
  m'(t)&=\varphi(t)[f(t-\tau)\varphi(t-\tau)-f(t)\varphi(t)]-\psi(t)[f(t-\tau)\psi(t-\tau)-f(t)\psi(t)]\\
  &=f(t-\tau)[\varphi(t)\varphi(t-\tau)-\psi(t)\psi(t-\tau)]-f(t)[\varphi(t)^2-\psi(t)^2]\\
  &\quad \pm f(t-\tau) \varphi(t)\psi(t-\tau) \\
  & =f(t-\tau)[\varphi(t)m(t-\tau)+m(t)\psi(t-\tau)]-f(t)[\varphi(t)^2-\psi(t)^2]\\
   & =f(t-\tau)[\varphi(t)m(t-\tau)+m(t)\psi(t-\tau)]-f(t)[\varphi(t)+\psi(t)]m(t), 
\end{align*}  
and then $|m'(t)|\leq 4M$ for all $t\geq t_0$. 

Furthermore, for all $t\geq t_0$ there is a triangle of height $|m(t)|$ and base 
$$|m(t)|/(4M)$$
 contained in  the graphic of $|m |:[t,t+\tau]\to\R$. Note that the base is smaller than $\tau$ by \eqref{cotaM}. Therefore,  we obtain 
\begin{equation*}
  \frac{|m(t)|^2}{8M}\leq \int_{t}^{t+\tau}|m(h)|\, dh \leq \frac{1}{\inf f}\int_{t}^{t+\tau}f(h)|m(h)|\, dh\leq  \frac{ M(t-t_0)\beta^{(t-t_0)/\tau-1}}{\inf f}
\end{equation*}   
and the result is proved.
\end{proof}

Finally, we can prove the last lemma of this section.  

\begin{lem}\label{lemaprincipal}
Let  $\tau\geq 0$ and  $\varphi,\,\psi:[-\tau,\infty)\to(0,1]$ be functions satisfying  
\begin{align*}
  \varphi(t)&=e^{-\int^t_{t-\tau}f(h)\varphi(h)\, dh},\\
  \psi(t)&=e^{-\int^t_{t-\tau}g(h)\psi(h)\, dh}
\end{align*}
for all $t\geq  0$. Given positive constants  $\eta$  and $T$, there exist another positive constants $\varepsilon$ and $\tilde T $  such that if 
$$|f(t)-g(t)|<\varepsilon$$
for all $t\in [t_0,t_0+\tilde T+T]$ (for any $t_0\geq 0$), then 
$$|\varphi(t)-\psi(t)|<\eta$$
for all  $t\in [t_0+\tilde T,t_0+\tilde T+T]$. 
\end{lem}

\begin{proof}
Firstly, note that if $\tau=0$ then $\varphi(t)=\psi(t)=1$ for all $t\geq 0$ and there is nothing to prove. Therefore, for the rest of the proof assume that $\tau>0$. 

Since
$$3M \sqrt{\frac{ x}{ \inf f}} \left(1-e^{-M\tau}\right)^{x/(2\tau)-1/2} $$
goes to zero when $x$ goes to infinity, fix a constant $\tilde T>0$ such that 
$$3M \sqrt{\frac{ x}{ \inf f}} \left(1-e^{-M\tau}\right)^{x/(2\tau)-1/2} <\eta/2$$
when $x\geq \tilde T$. Also fix $\varepsilon>0$ such that
$$\varepsilon<\frac{ \eta e^{-2M(T+\tilde T)}}{4\tau}.$$
Therefore fix any $t_0\geq 0$, and assume  
$$|f(t)-g(t)|<\varepsilon$$
for all  $t\in [t_0,t_0+\tilde T+T]$. Let $\tilde \psi:[t_0-\tau,\infty)\to(0,1]$ be   the function given by Lemma  \ref{hermana} satisfying 
$$\tilde \psi(t)=e^{-\int_{t-\tau}^tg(h)\tilde\psi(h)\, dh}$$
for all  $t>t_0 $,  and 
$$|\varphi(t)-\tilde \psi(t)|<2\varepsilon\tau e^{2M(t-t_0)} <\eta/2$$
for all  $t\in [t_0,t_0+\tilde T+T]$. Using Lemma  \ref{lemacrucial} for  $\psi$  and  $\tilde \psi$, we get 
\begin{align*}
  |\varphi(t)-\psi(t)|& \leq |\varphi(t)-\tilde \psi(t)|+|\tilde\psi(t)-\psi(t)|\\
  &<\eta/2+3M \sqrt{\frac{ (t-t_0)}{\inf f}} \left(1-e^{-M\tau}\right)^{(t-t_0)/(2\tau)-1/2} \\
  &<\eta
\end{align*}
when  $t\in [t_0+\tilde T,t_0+\tilde T+T]$. This proves the lemma.
\end{proof}

\begin{rem}
The fact that $\tilde T$ and $\varepsilon$ do not depend  on $t_0$ in the previous lemma  is crucial to prove Theorem \ref{persistencia}.
\end{rem}

\subsection{Proof of Theorem \ref{persistencia}}

  Set $f(t)= p(z(t ))$ and $g(t)=p(s(t ))$. Since by Lemma   \ref{funcionz} we know that $z$ and $s$ are upper bounded we set 
 $$M=\max_{t\geq 0}\{p(z(t)),p(s(t))\}.$$  
 
\begin{proof}[Proof of necessary conditions for persistence.]
To prove the first half of the theorem we follow the ideas developed in \cite[Lemma 3]{ellermeyer2001persistence}. Therefore assume the System \ref{sistemaprincipal} is persistent and fix a solution $(s,x)$ with non-null initial conditions. Recalling Remark \ref{observaciondefinicion}, there exists a positive constant  $\delta$ such that $x(t)\geq \delta$ for all  $t\geq \tau$.  By Lemma   \ref{funcionz} there is  $t_0>2\tau$ such that if   $t\geq t_0$, then 
\begin{equation*}
  |s(t-\tau)+x(t-\tau)+y(t-\tau)-z(t-\tau)|<\delta/2. 
\end{equation*}
Since $y(t-\tau)\geq 0$, we get
$$z(t-\tau)> s(t-\tau)+x(t-\tau)+y(t-\tau)-\delta/2 \geq s(t)+\delta/2$$
for all $t\geq t_0$. Since $p'>0$, there is a constant  $K>0$ which is a lower bound for  $p'(\xi)$ when $\xi\leq \max_{t\geq 0}\{z(t),s(t)\}$. Therefore, by the mean value theorem, it happens that  
$$p(z)-p(s)=p'(\xi)(z-s)\geq K(z-s),$$
and then 
$$p(z(t-\tau))>p(s(t-\tau))+ K\delta/2  $$
for all $t\geq t_0$. 

Now consider $\varphi$ defined in \eqref{funcionphi} and $\psi$ defined in \eqref{funcionpsi}. By  Lemma \ref{propiedadphi} and Lemma  \ref{vuelta},   there exist positive constants $\eta$ and $\tilde T$ such that 
\begin{equation*}
  \int_{t_1}^{t_2}p(z(t-\tau ))\varphi(t-\tau)\, dt\geq \int_{t_1}^{t_2} \left(p(s(t-\tau ))\psi(t-\tau)+2\eta \right)\, dt
\end{equation*}
for all  $t_2-\tilde T\geq t_1\geq t_0$. Moreover, since
$$\frac{d}{dt}\ln(x(t))= -D(t)+p(s(t-\tau))\psi(t-\tau), $$
then 
$$\ln(x(t_2 ))-\ln(x(t_1 ))=\int_{t_1 }^{t_2 } \left(p(s(t-\tau))\psi(t-\tau )-D(t)\right)\, dt.$$
Note that by Lemma \ref{funcionz}, and the hypothesis of persistence,  $x(t)$ is upper  and lower bounded by positive constants  for all $t\geq \tau$. Then,  fix $T>\max\{t_0,\tilde T\}$ large enough for  
$$\left|\frac{\ln(x(t_2 ))-\ln(x(t_1 ))}{T}\right|<\eta $$ 
for all $t_2\geq t_1\geq \tau.$

Therefore, we get the constants $\eta$ and $T$ as we need:
\begin{align*}
  \int_{t_1}^{t_2}\left(p(z(t-\tau))\varphi(t-\tau)-D(t)\right)\, dt&\geq  \int_{t_1}^{t_2}\left(p(s(t-\tau))\psi(t-\tau)+2\eta-D(t)\right)\, dt\\
  &=\ln(x(t_2))-\ln(x(t_1))+\int_{t_1}^{t_2} 2\eta \, dt\\
  &= \int_{t_1}^{t_2}  \left(\frac{\ln(x(t_2))-\ln(x(t_1))}{t_2-t_1}+2\eta\right)\, dt\\
  &>  \int_{t_1}^{t_2}  \eta \, dt
\end{align*}
if $t_1>T$ and $t_2-t_1>T$. This proves the first half of the theorem.

\end{proof}

\begin{proof}[Proof of sufficient conditions for persistence.]

Fix  $\eta$ and  $T $ as in the statement, we need to show the persistence. Therefore fix $(s,x)$  solution of System \ref{sistemaprincipal} with non-null initial conditions. Since $\si$ is lower bounded by a positive constant and $z(0)>0$ we have that $p(z)$ is lower bounded by a positive constant. Then consider the  constants $\eta/(4M)$ and $T$, by Lemma   \ref{lemaprincipal} and  for any $t_0\geq 0$,  there are positive  constants $\tilde T$ and  $\varepsilon$ (independent of $t_0$) such that if 
$$|p(z(t ))-p(s(t ))|<\varepsilon$$
for all  $t\in[t_0,t_0+T+\tilde T]$, then 
\begin{equation}\label{primeradesigualdad}
  |\varphi(t)-\psi(t)|<\frac{\eta}{4M }  
\end{equation} 
for all $t\in[t_0+\tilde T,t_0+T+\tilde T].$ Also we can fix $\varepsilon$ small enough so that  
\begin{equation}\label{segundadesigualdad}
  \varepsilon\leq \eta/4.
\end{equation}
Furthermore, by Lemma \ref{funcionz} there exists $t_0>\tau$ such that if $t>t_0$, then 
\begin{equation*}
  |s(t)+x(t)+y(t)-z(t)|< \frac{\varepsilon}{2L} 
\end{equation*}
where  $L$ is a positive upper bound for   $p'(z)$ and  $p'(s)$ (it exists by the same reason that $K$ exists in the previous proof). 

Now define
\begin{align*}
  \delta&=\min\left\{\min_{t\in [\tau ,t_0+T+\tilde T +2\tau ]}x(t), \frac{\varepsilon  e^{- (\max D) (T+\tilde T+2\tau)}}{2L  (1+M\tau)}\right\}>0,\\
  \mathcal{S}&=\left\{t\geq \tau: x(h)\geq \delta\text{ for all }h\in[\tau,t]\right\} 
\end{align*}
 where $\max D$ means the maximum of $D$. By definition  $\mathcal{S}$ is not empty since   $t_0+T+\tilde T +2\tau \in \mathcal{S}$. Denoting by $t^*=\sup\mathcal{S}$, we claim that $t^*=\infty$. For contradiction assume that 
$$ t_0+T+\tilde T+2\tau <t^*<\infty.$$
 Then by continuity, we get 
\begin{equation}\label{contradicciondelta}
  x(t^*)=\delta  .
\end{equation}
By definition,  it happens that 
$$\frac{d}{dt}\ln(x(t))\geq -D(t)x(t),$$
and then for $0\leq t\leq t^*  $, we get
$$ x(t^* )\geq x(t )e^{-\int_{t }^{t^*}D(r)\, dr}\geq x(t )e^{- (\max D) (t^* -t)} .$$
This implies that 
$$ x(t)\leq x(t^* )e^{ (\max D) (T+\tilde T+ 2\tau)} \leq \frac{\varepsilon}{2L(1+M\tau ) }$$
for all  $ t\in [t^*-T-\tilde T-2\tau,t^* ] $. Furthermore,  by   definition 
\begin{align*}
  y(t)&\leq \int_{t-\tau}^t x(h)p(s(h)) \,dh\leq  M \int_{t-\tau}^t x(h)  \,dh\leq\frac{ M \tau \varepsilon}{2L(1+M\tau )}
\end{align*}
for all  $t\in [t^*-T-\tilde T-\tau,t^* ]$. Therefore 
\begin{align*}
  |z(t)-s(t)|&\leq |z(t)-s(t)-x(t)-y(t)|+x(t)+y(t)\\
  &\leq \frac{\varepsilon}{2L}+\frac{\varepsilon}{2L(1+M\tau )}+\frac{ M \tau \varepsilon}{2L(1+M\tau )}\\
  &\leq \varepsilon/L 
\end{align*}
for all  $t\in [t^*-T-\tilde T-\tau,t^* ]$. By the mean value theorem, we get 
$$|p(z(t))-p(s(t ))|\leq L|z(t )-s(t )|<\varepsilon$$
for all $t\in [t^*-T-\tilde T-\tau,t^*]$. 
Then, by   \eqref{primeradesigualdad} and \eqref{segundadesigualdad}, we obtain
\begin{align*}
  |p(z(t ))\varphi(t)-p(s(t ))\psi(t)|&\leq |p(z(t ))\varphi(t)-p(z(t))\psi(t)| \\
  &\quad +|p(z(t ))\psi(t)-p(s(t ))\psi(t)|\\
  &\leq p(z(t ))|\varphi(t)-\psi(t)| +|p(z(t ))-p(s(t ))|\psi(t)\\
  &\leq M\frac{\eta}{4M}+\frac{\eta}{4 }\\
  &=\eta/2
\end{align*}
for all $t\in[t^*-T-\tau,t^*]$.  Finally, by hypothesis we have that 
$$\int_{t^*-T}^{t^*}\left(p(z(t-\tau))\varphi(t-\tau)-D(t)-\eta\right)\, dt\geq 0,$$
which implies that 
\begin{align*}
  \int_{t^*-T}^{t^*}\left(p(s(t-\tau))\psi(t-\tau)-D(t) \right)\, dt&\geq \int_{t^*-T}^{t^*}\left(p(z(t-\tau))\varphi(t-\tau)- \frac{\eta}{2} -D(t)\right) \, dt 
\end{align*}
is positive and then
\begin{align*}
  x(t^*)&=x(t^*-T)e^{\int_{t^*-T}^{t^*}\left(p(s(h-\tau))\psi(  h-\tau )-D(h)\right)\, dh} \\
  &\geq \delta e^{\int_{t^*-T}^{t^*}\left(p(s(h-\tau))\psi(  h-\tau )-D(h)\right)\, dh }\\
  &>\delta.
\end{align*}
This is a contradiction of \eqref{contradicciondelta}, then   $t^*=\infty$, and the theorem is proved.
\end{proof}

\section{Remaining proofs}\label{otraspruebas}

\subsection{Proof of Theorem \ref{persistenciaperiodica}}

We now assume that $\si$ and $D$ follow the hypothesis of Theorem \ref{persistenciaperiodica}. We need two  lemmas in order to prove this result,  although the first is a classical result, we provide the details here for the convenience of the reader.

\begin{lem}
There is a unique $\omega$-periodic solution $z$ of \eqref{derivadaz}.
\end{lem}

\begin{proof}
We  consider $\si$ and $D$ $\omega$-periodic with $\R$-domain. Then define
$$z(t)=\int_{-\infty}^t \si(h)D(h)e^{-\int_h^tD(r)\, dr}\,dh.$$
This function is well defined since $\si$ is upper bounded and $z$ satisfies \eqref{derivadaz}. Furthermore
\begin{align*}
  z(t+\omega)&=\int_{-\infty}^{t+\omega} \si(h)D(h)e^{-\int_h^{t+\omega}D(r)\, dr}\,dh\\
  &=\int_{-\infty}^{t} \si(h+\omega)D(h+\omega)e^{-\int_{h+\omega}^{t+\omega}D(r)\, dr}\,dh\\
  &=\int_{-\infty}^{t} \si(h )D(h )e^{-\int_{h }^{t}D(r+\omega)\, dr}\,dh\\
  &=z(t),
\end{align*}
and then it is $\omega$-periodic. Finally, assume there is another $\omega$-periodic solution $\tilde z$. Then 
$$\frac{d}{dt}z(t)-\tilde z(t)=-D(t)(z(t)-\tilde z(t)),$$
and
$$|z(t)-\tilde z(t)|=|z(0)-\tilde z(0)|e^{-\int_0^tD(r)\,dr}.$$
Since both are $\omega$-periodic they must be the same function.
\end{proof} 

\begin{lem}
Given $z$ $\omega$-periodic there is a unique $c$ (up to a constant factor) defined in \eqref{funcionc} such that $\varphi$ defined  in   \eqref{funcionphi} is $\omega$-periodic.
\end{lem}

\begin{proof}
Fix $c$ with $c(0)>0$ and let $\psi:[-\tau,\infty)$  be defined as
$$\psi(t)=\frac{c(t)}{c(t+\tau)}e^{-\int_{t }^{t+\tau}D(r)\,dr} .$$ 
  For all $n\in\N$ define $\psi_n\in C([0,\omega],(0,1])$ by
$$\psi_n(t)=\psi(t+n\omega) $$
with the norm $\|\psi_n\|=\max_{t\in[0,\omega]}|\psi_n(t)|.$ Therefore we claim that $(\psi_n)_{n\in\N}$ is a Cauchy sequence. To see this fix $\varepsilon>0$. Let $n_0\in\N$ be such that
$$3M\sqrt{\frac{   n_0\omega}{\inf f}}\left(1-e^{-M\tau}\right)^{n_0\omega/(2\tau)-1/2}<\varepsilon $$
and assume that $n_0\leq m< n$. By Lemma \ref{propiedadphi} we have that 
\begin{align*}
  \psi(t)&=e^{-\int^t_{t-\tau}p(z(h))\psi(h)\, dh},\\
  \psi(t+(n-m)\omega)&=e^{-\int^t_{t-\tau}p(z(h))\psi(h+(n-m)\omega)\, dh}  
\end{align*}
for all $t\geq 0$ since $p(z)$ is $\omega$-periodic. By Lemma \ref{lemacrucial}, if   $t\in[0,\omega]$, then
\begin{align*}
  |\psi_m(t)-\psi_n(t)|&=|\psi (t+m\omega)-\psi (t+m\omega+(n-m)\omega)|\\
  &<3M\sqrt{\frac{    t+m\omega }{\inf f}}\left(1-e^{-M\tau}\right)^{(t+m\omega )/(2\tau)-1/2}\\
  &<\varepsilon,
\end{align*}
and then $\|\psi_m -\psi_n\|<\varepsilon$, which proves it is a Cauchy sequence. Then define $\varphi:[-\tau,\infty)\to[0,1]$ by
$$\varphi(t)=\lim_{n\to\infty}\psi_n (t ).$$
It is $\omega$-periodic since
\begin{align*}
  |\varphi(t)-\varphi(t+\omega)|&\leq |\varphi(t)-\psi_n(t)|+|\psi_n(t)-\psi_{n+1}(t)|+|\psi_n(t+\omega)-\varphi(t+\omega)|,
\end{align*}
which tends to zero as $n$ goes to infinity. Note that 
\begin{align*}
  \left| \int_{t-\tau}^tp(z(h))\psi_n(h)\,dh -\int_{t-\tau}^tp(z(h))\varphi(h)\,dh \right|&\leq M\left| \int_{t-\tau}^t \psi_n(h)-\varphi(h)\,dh \right|\\
  &\leq M\tau \|\psi_n-\varphi\|,
\end{align*}
and therefore
\begin{align*}
  \left|\varphi(t)-e^{-\int_{t-\tau}^tp(z(h))\varphi(h)\,dh}\right|&\leq |\varphi(t)-\psi_n(t)|\\
  &\quad +\left|e^{-\int_{t-\tau}^tp(z(h))\psi_n(h)\,dh} -e^{-\int_{t-\tau}^tp(z(h))\varphi(h)\,dh}\right|\\
  &\leq (1+M\tau)\|\psi_n-\varphi\|,
\end{align*}
which again tends to zero as $n$ goes to infinity. Then $\varphi$ satisfies 
\begin{equation}\label{propiedadlimite}
  \varphi(t)=e^{-\int_{t-\tau}^tp(z(h))\varphi(h)\,dh}
\end{equation} 
for all $t$.

It remains to prove there is a unique $c$ (up to a constant factor) that generates $\varphi$. Firstly define
$$c(t)=e^{\int_{0}^t(-D(h)+p(z(h-\tau))\varphi(h-\tau))\,dh}.$$
Then
\begin{align*}
  \frac{c(t)}{c(t-\tau)}e^{-\int_{t}^{t+\tau}D(h)\, dh}&=\frac{e^{\int_{0}^t(-D(h)+p(z(h-\tau))\varphi(h-\tau))\,dh-\int_{t}^{t+\tau}D(h)\, dh}}{e^{\int_{0}^{t+\tau}(-D(h)+p(z(h-\tau))\varphi(h-\tau))\,dh}}\\
  &=e^{-\int_{t}^{t+\tau} p(z(h-\tau))\varphi(h-\tau)\,dh}\\
  &=\varphi(t)
\end{align*} 
by property \eqref{propiedadlimite}. Suppose now there are $c_1$ and $c_2$ that generates $\varphi_1$ and $\varphi_2$, respectively, with both $\omega$-periodic. Since both are periodic, by Lemma \ref{lemacrucial}, we get that $\varphi_1=\varphi_2$ and then
\begin{align*}
  \frac{d}{dt}\frac{c_1(t)}{c_2(t)}&=\frac{-D(t)c_1(t)+c_1(t-\tau)p(z(t-\tau))e^{-\int_{t-\tau}^t D(r)\,dr}}{c_2(t)}\\
  &\quad- c_1(t) \frac{-D(t)c_2(t)+c_2(t-\tau)p(z(t-\tau))e^{-\int_{t-\tau}^t D(r)\,dr}}{c_2^2(t)}\\
  &=\frac{c_1(t)p(z(t-\tau))}{c_2(t)}\left(\varphi_1(t)-\varphi_2(t)\right)\\
  &=0
\end{align*}
and the lemma is proved.
\end{proof}
 
\begin{proof}[Proof of Theorem \ref{persistenciaperiodica}] 
With the previous two lemmas we get the existence and uniqueness of $z$ and $\varphi$. We claim that   there are positive constants   $\eta$ and $T$ such that 
$$
  \int_{t_1}^{t_2} p(z(t-\tau)) \varphi(t-\tau)\, dt> \int_{t_1}^{t_2}\left( D(t)+\eta \right) \, dt
$$
for all $t_1> T$, $t_2-t_1> T$  if and only if
\begin{equation}\label{desigualdadpromedio}
  \langle p(z ) \varphi \rangle>\langle  D \rangle. 
\end{equation}
To see this first fix $\eta$ and $T$ as in the statement and  consider $N\in\N$ such that $N\omega>T$. Then
\begin{align*}
  \langle p(z) \varphi \rangle& = \frac{1}{N\omega}  \int_{t_1-\tau }^{t_1+N\omega -\tau} p(z(t  )) \varphi(t )\, dt\\
  & = \frac{1}{N\omega}  \int_{t_1 }^{t_1+N\omega } p(z(t -\tau)) \varphi(t-\tau)\, dt \\
  &>\frac{1}{N\omega}  \int_{t_1}^{t_1+N\omega}\left( D(t)+\eta \right) \, dt\\
  &>\langle  D \rangle.
\end{align*}
For the other half, if we assume \eqref{desigualdadpromedio}, then there is $\eta>0$ such that
$$ \langle p(z) \varphi \rangle>\langle  D \rangle+2\eta.$$
Fix $n\in\N$ such that $n\geq  \max D /\eta+1 $, and define $T=n\omega $. If $t_2-T> t_1>T$ there is $m\geq n$ such that $t_2\in[t_1+m\omega,t_1+(m+1)\omega].$ Then
\begin{align*}
  \int_{t_1}^{t_2} p(z(t-\tau)) \varphi(t-\tau)\, dt&\geq  \int_{t_1}^{t_1+m\omega} p(z(t-\tau)) \varphi(t-\tau)\, dt\\
  &=m\omega\langle p(z) \varphi \rangle\\
  &>m\omega\left(\langle  D \rangle+2\eta\right)\\
  &=\int_{t_1}^{t_2} \left(D(t)+ \eta\right)  \, dt+m\omega\eta  -\int_{t_1+m\omega}^{t_2}\left( D(t)+ \eta \right)  \, dt\\
  &\geq \int_{t_1}^{t_2}\left( D(t)+ \eta \right) \, dt+n\omega\eta+ -\omega( \max D +\eta)\\
   &\geq \int_{t_1}^{t_2}\left( D(t)+ \eta \right) \, dt.
\end{align*}
Finally, applying  Theorem \ref{persistencia} we get  the desired result.
\end{proof}

\subsection{Proof of Proposition \ref{implicacion}}
To end this article we provide the last proof.

\begin{proof} 
Since by hypothesis System \ref{sistemaprincipal} is persistent, by Theorem \ref{persistencia} we know that there are positive constants $\eta$ and $T$ such that
$$\int_{t_1}^{t_2} p(z)\varphi\,dt=p(z)\varphi(t_2-t_1)>\int_{t_1}^{t_2}\left(D(t)+\eta \right)\,dt$$
for all $t_1>T$ and $t_2>t_1+T$ with $\varphi$ constant such that $\varphi=e^{-\tau p(z)\varphi}.$
We can assume that  $T$ is large enough so that  
$$ T\geq \frac{2}{  \eta}\max_{t\geq 0}\{p(z),D(t) /\varphi\}. $$ 
Also note that 
$$e^{-a}-e^{-b} \geq e^{-b}(b-a),$$
no matter the sing of $b-a$. Therefore 
\begin{align*}
  \int_{t_1}^{t_2}\left(e^{-\int_{t-\tau}^tD(r)\,dr}-\varphi\right)\, dt&=\int_{t_1}^{t_2}\left(e^{-\int_{t-\tau}^tD(r)\,dr}-e^{-\tau p(z)\varphi}\right)\, dt\\
  &\geq\int_{t_1}^{t_2}e^{-\tau p(z)\varphi}\left( \tau p(z) \varphi  -\int_{t-\tau}^tD(h)\,dh\right)\, dt\\
  &=\varphi \int_{t_1}^{t_2} \int_{t-\tau}^t  \left( p(z)\varphi  - D(h)\right)\,dh  dt\\
  &=\varphi \left(\int_{t_1-\tau}^{t_1} \left(p(z) \varphi  - D(h)\right)\int_{t_1}^{h+\tau} 1\, dtdh\right.\\
  &\quad +\int_{t_1 }^{t_2-\tau} \left(p(z) \varphi  - D(h)\right)\int_{h}^{h+\tau} 1\, dtdh\\
  &\quad\left.+\int_{t_2-\tau}^{t_2} \left(p(z) \varphi  - D(h)\right)\int_{h}^{t_2} 1\, dtdh\right)\\
  &\geq \varphi\tau \int_{t_1}^{t_2}\left( p(z)\varphi-D(h)\right)\,dh\\
  &\quad -2\varphi\tau  \max_{t\geq 0}\{p(z),D(t) /\varphi\}  \\
  &=\varphi\tau\left(\eta T -2  \max_{t\geq 0}\{p(z),D(t) /\varphi\}  \right) \\
  &\geq  0.  
\end{align*}
Finally,
$$\int_{t_1}^{t_2} p(z)e^{-\int_{t-\tau}^tD(r)\,dr}\,dt\geq\int_{t_1}^{t_2} p(z)\varphi\,dt >\int_{t_1}^{t_2}\left(D(t)+\eta \right)\,dt,$$
and the proposition is proved.
\end{proof}

 \section*{Acknowledgments}  The author would like to thank  Professor P. Amster for his valuable remarks on this manuscript.    This work was partially supported by  Conicet  under grant PIP 11220200100175CO and by project TOMENADE [MATH-AmSud, 21-MATH-08].
%
%
%
%
%
%
%

\addcontentsline{toc}{chapter}{Bibliograf\'ia}
\bibliography{biblio}
\bibliographystyle{plain}
\end{document}